\documentclass[11pt,a4paper]{amsart}

\usepackage{fancyhdr,stmaryrd}
\usepackage{amssymb} 
\usepackage{combelow} 
\usepackage{tensor} 
\usepackage{hyperref} 
\usepackage[all]{xy} 
\usepackage{mathrsfs} 
\usepackage{indentfirst} 
\usepackage{mathtools} 
\usepackage{titlesec} 
\usepackage[a4paper]{geometry}
\usepackage{xstring,xifthen}
\SetSymbolFont{stmry}{bold}{U}{stmry}{m}{n}
\usepackage{bm}
\usepackage{enumerate}
\usepackage{amsfonts,amsmath}

\newtheorem{theorem}{Theorem}[section]
\newtheorem{corollary}[theorem]{Corollary}
\newtheorem{lemma}[theorem]{Lemma}
\newtheorem{proposition}[theorem]{Proposition}
\theoremstyle{definition}
\newtheorem{definition}[theorem]{Definition}
\newtheorem{example}[theorem]{Example}
\newtheorem{remark}[theorem]{Remark}

\newcommand{\C}{\mathcal{C}}
\renewcommand{\c}{\mathrm{c}}
\renewcommand{\d}{\mathrm{d}}
\newcommand{\m}{\textrm{-}}
\newcommand{\f}{\bm{f}}
\renewcommand{\i}{\bm{i}}
\newcommand{\s}{\bm{s}}
\renewcommand{\S}{\mathbb{\mathcal{S}}}
\renewcommand{\H}{\mathcal{H}}
\newcommand{\G}{\mathcal{G}}
\newcommand{\D}{\mathcal{D}}
\newcommand{\id}{\mathrm{id}}
\newcommand{\Id}{\mathrm{Id}}
\newcommand{\oM}{\overline{M}_S}
\newcommand{\Ob}{\mathrm{Ob}}
\newcommand{\op}{\mathrm{op}}
\newcommand{\obc}{\mathrm{Ob}\,\mathcal{C}}
\newcommand{\obd}{\mathrm{Ob}\,\mathcal{D}}
\newcommand{\Alg}{\mathrm{Alg}}
\newcommand{\Hom}{\mathrm{Hom}}
\newcommand{\Ind}{\mathrm{IndT}}
\newcommand{\IndP}{\mathrm{IndP}}
\newcommand{\Mor}{\mathrm{Mor}\,}

\makeatletter
\newcommand{\@ndereq}[2]{%
  \vtop{
    \lineskiplimit\maxdimen
    \lineskip-.5\p@
    \ialign{$\m@th#1\hfil##\hfil$\crcr\sim\crcr#2\crcr}%
  }%
}
\newcommand{\sims}{\mathrel{\mathpalette\@ndereq{\scriptstyle \s}\relax}}
\makeatother
  
\newcommand{\titlename}	
{Skew category algebras, twisted tensor product algebras and induction of precosheaves}
\newcommand{\shorttitlename}
{Skew category algebras and induction}

\newcommand{\authorname}      {Tiberiu Coconeț$^1$, Virgilius-Aurelian Minuță$^2$ \and Constantin-Cosmin Todea$^3$}
\newcommand{\pdfauthorname}   {Tiberiu Coconet, Virgilius-Aurelian Minuta, Constantin-Cosmin Todea}
\newcommand{\shortauthorname} {T. Coconeț, V. A. Minuță \and C.-C. Todea}

\newcommand{\universitynameA}  {$^{1}$Babeș-Bolyai University}
\newcommand{\facultynameA}     {Faculty of Economics and Business Administration}
\newcommand{\departmentnameA}  {}
\newcommand{\addressA}  	   {T. Mihali 58-60, RO-400591, Cluj-Napoca, Romania}

\newcommand{\universitynameB}  {$^{2}$Babeș-Bolyai University}
\newcommand{\facultynameB}     {Faculty of Mathematics and Computer Science}
\newcommand{\departmentnameB}  {Department of Mathematics}
\newcommand{\addressB}  	   {M. Kogălniceanu 1, RO-400084, Cluj-Napoca, Romania}

\newcommand{\universitynameC}  {$^{3}$Technical University of Cluj-Napoca}
\newcommand{\facultynameC}     {}
\newcommand{\departmentnameC}  {Department of Mathematics}
\newcommand{\addressC}  	   {G. Barițiu 25, RO-400027, Cluj-Napoca, Romania}

\newcommand{\emailaddressA}    {tiberiu.coconet@math.ubbcluj.ro}
\newcommand{\emailaddressB}    {virgilius.minuta@ubbcluj.ro}
\newcommand{\emailaddressC}    {constantin.todea@math.utcluj.ro}

\newcommand{\articleabstract}{We show that a skew category algebra can be embedded into a twisted tensor product algebra. We investigate the extension of some concepts of Puig and Turull from group algebras to category algebras and their behavior with respect to skew category algebras.}

\newcommand{\msc}{16SXX, 18AXX, 18A22}

\newcommand{\keywordterms}{Category, Skew, Algebra, Twisted, Induction, Precosheaf}

\def\depA{\departmentnameA}
\StrLen{\depA}[\depAlen]
\def\depB{\departmentnameB}
\StrLen{\depB}[\depBlen]
\def\depC{\departmentnameC}
\StrLen{\depC}[\depClen]
\def\facA{\facultynameA}
\StrLen{\facA}[\facAlen]
\def\facB{\facultynameB}
\StrLen{\facB}[\facBlen]
\def\facC{\facultynameC}
\StrLen{\facC}[\facClen]

\newcommand{\institutionA}{
\universitynameA\\
\ifthenelse{\facAlen>0}{\facultynameA\\}{}
\ifthenelse{\depAlen>0}{\departmentnameA\\}{}
\addressA}
\newcommand{\institutionB}{
\universitynameB\\
\ifthenelse{\facBlen>0}{\facultynameB\\}{}
\ifthenelse{\depBlen>0}{\departmentnameB\\}{}
\addressB}
\newcommand{\institutionC}{
\universitynameC\\
\ifthenelse{\facClen>0}{\facultynameC\\}{}
\ifthenelse{\depClen>0}{\departmentnameC\\}{}
\addressC}
\hypersetup{colorlinks = true, linkcolor = blue, anchorcolor =red, citecolor = blue, filecolor = red, urlcolor = blue, pdfauthor=\pdfauthorname, pdftitle=\titlename, pdfkeywords=\keywordterms, pdfsubject=\articleabstract}
\titleformat{\section}{\Large\bfseries}{\thesection}{1em}{}
\titleformat{\subsection}{\large\it}{\thesubsection}{1em}{}
\fancypagestyle{mypagestyle}{
  \fancyhf{}
  \fancyhead[OC]{\textit{\shorttitlename}}
  \fancyhead[EC]{\textit{\shortauthorname}}
  \fancyfoot[C]{\medskip $\leftslice\,$\thepage$\,\rightslice$ }
  
}
\fancypagestyle{firstpagestyle}{
  \fancyhf{}
  \fancyfoot[C]{\medskip $\leftslice\,$\thepage$\,\rightslice$ }
  
}
\title[\shorttitlename]{\LARGE{\titlename}}
\author[\shortauthorname]{\large{\authorname}
\medskip\\
{\footnotesize \institutionA\medskip\\\institutionB\medskip\\\institutionC\medskip\\
$\begin{array}{l}
\text{email}^1\text{: \texttt{\href{mailto:\emailaddressA}{\emailaddressA}}}\\
\text{email}^2\text{: \texttt{\href{mailto:\emailaddressB}{\emailaddressB}}}\\
\text{email}^3\text{: \texttt{\href{mailto:\emailaddressC}{\emailaddressC}}}
\end{array}$
}}
\begin{document}
\begin{abstract}
\articleabstract\\[0.1cm]
\textsc{MSC 2010.} \msc\\[0.1cm]
\end{abstract}
\begingroup
\def\uppercasenonmath#1{} 
\let\MakeUppercase\relax 
\maketitle
\endgroup
\thispagestyle{firstpagestyle}


\section{Introduction} \label{sec1}
Let $\mathcal{C}$ be a finite category (the class of morphisms $\Mor \C$ is a finite set) and $k$ be a field. Even though  some of the results are true for small categories with finitely many objects and over unital commutative rings, for our purpose, we prefer working with finite categories and fields.  If $f:x\to y$ is a morphism in $\mathcal{C}$, we denote by $\mathrm{d}(f)$ and $\mathrm{c}(f)$ the domain $x$ and the codomain $y$, respectively. By a $k$-algebra we mean a unital associative $k$-algebra and we assume that a $k$-subalgebra  of  a larger $k$-algebra has the same identity, if otherwise not stated. Let $k\textrm{-}\Alg$ be the category of unital associative $k$-algebras with unital $k$-algebra homomorphisms and $k\textrm{-}\mathrm{Vect}$ be the category of $k$-vector spaces. The identity morphism of an object $x\in\obc$ is denoted by $1_x.$ The reader should understand the difference between
$1_x, x\in\obc$ and the identity $1_A$ of a $k$-algebra $A;$ the identity $1_k$ of $k$ is denoted simply $1.$

The \textit{category algebra} $k\C$ (also named the category convolution algebra) is the $k$-vector space with a basis the set $\Mor \C.$ The multiplication is given by
 \[g\cdot f=\left\{\begin{array}{ll}
g\circ f,& \text{if } \mathrm{d}(g)=\mathrm{c}(f)\\
0,& \text{otherwise,}
\end{array}\right.\]
for any $f,g\in\Mor \C.$
The identity element of this algebra is $1_{k\C}=\sum_{x\in\obc}1_x.$ By $(k\textrm{-}\mathrm{Vect})^{\C}$ we denote the category of covariant functors from $\C$ to the category $k\textrm{-}\mathrm{Vect}$ of finite dimensional $k$-vector spaces. A well-known result of Mitchell (see \cite{art:Mitchell1972}) states that, since $\C$ has finitely many objects then the category $(k\textrm{-}\mathrm{Vect})^{\C}$ is equivalent to the category $k\C\textrm{-}\mathrm{mod},$ of finitely generated $k\C$-modules. We recall the explicit identification between $k\C\textrm{-}\mathrm{mod}$ and $(k\textrm{-}\mathrm{Vect})^{\C}$ in Section \ref{subsec:31}.

\goodbreak
By a \textit{precosheaf} of $k$-algebras on $\C$  we understand  a covariant functor $R:\C\to k\textrm{-}\Alg$. Skew category algebras are constructions that generalize skew group algebras in a slightly different manner than smash products. The skew category algebra $R[\C]$ on $\C$ with respect to $R$ (that will be explicitly presented in Section \ref{subsec:21}) becomes, when $\C$ is a poset, the construction (denoted by $R!$) obtained by Gerstenhaber and Schack \cite{art:Gers1988}. Historically, the first appearance of skew category algebras seems to be in \cite{OiLu} as a particular case of the so-called ``category crossed product''. There is also a more recent interest for these constructions in two papers of V. V. Bavula \cite{Bav1,Bav2}, but also of M. Wu and F. Xu \cite{art:Wu2022}.

\textit{Twisted tensor product algebras} were introduced by \v{C}ap, Schichl  and Van\v{z}ura in \cite[Definition 2.1.]{art:Cap1995} motivated by some unanswered questions  in non-commutative geometry. If $(A,\mu_A), (B,\mu_B)$ are two $k$-algebras then a $k$-linear map $\tau:B\otimes A\rightarrow A\otimes B$ is called a twisting map \cite[Proposition/Definition 2.3.]{art:Cap1995}  if 
\[\tau (b\otimes 1_A)=1_A\otimes b,\qquad  \tau (1_B\otimes a)=a\otimes 1_B\] 
and
\[\tau\circ(\mu_B\otimes\mu_A)=(\mu_A\otimes\mu_B)\circ (\mathrm{id}_A\otimes \tau\otimes \mathrm{id}_B)\circ (\tau\otimes\tau)\circ(\mathrm{id}_B\otimes\tau\otimes\mathrm{id}_A),\]
for any $a\in A,b\in B,$ where $\mathrm{id}_A,\mathrm{id}_B$ are the identity maps. If $\tau$ is a twisting map then $(A\otimes B, \mu_{\tau})$ is called a twisted tensor product algebra, where $$\mu_{\tau}:=(\mu_A\otimes\mu_B)\circ (\mathrm{id}_A\otimes \tau\otimes \mathrm{id}_B).$$ 

In Proposition \ref{prop:2_1} we construct a twisting map
\[\tau:k\C\otimes \left(\bigotimes_{x\in \obc}R(x)\right)\to \left(\bigotimes_{x\in \obc}R(x)\right)\otimes k\C\] hence, by Corollary \ref{cor:2.2}, there is a structure of twisted tensor product algebra 
$$\left(\left(\bigotimes\limits_{x\in\obc}R(x)\right)
\otimes k\C,\mu_{\tau}\right).$$ 

In our first main result we will show that any skew group algebra can be embedded into a twisted tensor product algebra, defined as above.
\begin{theorem}\label{thm:11} There is an injective homomorphism of  $k$-algebras
\[\Psi:R[\C]\to\left(\bigotimes_{x\in\obc}R(x)\right)\otimes_{\tau}k\C,\]
from the skew category algebra into the above twisted tensor product algebra.
\end{theorem}
\begin{corollary}
If $\C$ is a one object category (i.e. a monoid) then $\Psi$ becomes an isomorphism of $k$-algebras.
\end{corollary}

For an $H$-algebra ($H$ is a subgroup of a given group $G$) Turull defined in \cite{art:Tu2006} the induced $G$-algebra which he used for investigating  Clifford classes. On the other hand, Puig introduced the concept of general induction  for interior group-algebras in \cite{art:Puig1981,book:Puig1999}. 
In \cite[Definition 1.1]{art:Linckelmann2002} Linckelmann defines an induction for interior algebras generalizing  Puig's concept of induction. We use the terminology of \cite{art:Kuels2000}, which extends that of \cite[3.1]{art:Puig1981} to the so-called interior $A$-algebras (where $A$ is any $k$-algebra). 

A category algebra is not a bialgebra, see \cite{art:Xu2013}. Since it is only a weak bialgebra, we can not apply the machinery developed in \cite{art:CMT2017} to obtain an isomorphism between two types of inductions involving Hopf algebras (generalizing the group case inductions of Puig and Turull) through smash product, see \cite{art:Coc2009} for the group case. However, if we consider an injective on objects, surjective on morphisms covariant functor from a finite category $\D$ to a finite category $\C$ we are able to construct Puig's  and Turull's induction for category algebras.  

\goodbreak
Starting with Section \ref{sec3} we consider two finite categories $\C,\D$ and  a precosheaf of $k$-algebras on $\D,$ that will be denoted by $S:\D\rightarrow k\textrm{-}\Alg.$ 
We also need an injective on objects, surjective on morphisms covariant functor $\s:\D\to\C.$ 
In Section \ref{sec3} we define surjective Turull's induction of precosheaves of $k$-algebras $\Ind^{\C}_{\s,\D}(S):\C\rightarrow k\textrm{-}\Alg,$
see Definition \ref{def:37} and Proposition \ref{prop:38}.
 
In Section \ref{sec4} we continue with the above functor $\s,$ but we work with interior algebras \cite[Definition 1.1]{art:Linckelmann2002}. Recall that if $A$ is a $k$-algebra, an interior $A$-algebra is a $k$-algebra $C$ with a $k$-algebra homomorphism $\sigma:A\rightarrow C,$  denoted  $(C,\sigma).$ By assuming a condition on $\s$ (see Lemma \ref{lemma:422}, (\ref{eq:423})) we introduce the surjective Puig's induction $\IndP_{\s,\D}^{\C}(C),$ where $C$ is an interior $k\D$-algebra, see Definition \ref{def:44} and Proposition \ref{prop:45}.  We will finalize Section \ref{sec4} with the proof of the following theorem. But first we need some background on category graded algebras. A category graded algebra,  or  a $\C$-graded algebra,  is a $k$-algebra $A$ for which there exist the homogeneous components $A_f,$  $f\in \Mor(\C),$ such that $A=\bigoplus_{f\in \Mor(\C)}A_f$ and $A_f\cdot A_g\subseteq A_{f\circ g},$ for any $f,g\in \Mor(\C)$ such that $\mathrm{d}(f)=\mathrm{c}(g).$ A homomorphism of $\C$-graded algebras is a homomorphism of $k$-algebras that preserves the homogeneous components.

\begin{theorem}\label{thm:13}
Let $S:\D\to k\mathrm{-Alg}$ be a precosheaf of $k$-algebras on $\D$. Let $\s:\D\to\C$ be an injective on objects, surjective on morphisms covariant functor. The following statements hold:
\begin{enumerate}[\normalfont (i)]
\item The skew category algebra $S[\D]$ is an interior $k\D$-algebra. More precisely, there is a homomorphism of $\D$-graded  algebras
\[\sigma:k\D\to S[\D],\quad \sigma(f)=1_{S(\c(f))}f,\]
for any $f\in \Mor \D;$
\item If $\mathrm{(\ref{eq:423})}$ is true, then there is an isomorphism of $\C$-graded interior $k\C$-algebras
\[\IndP^{\C}_{\s,\D}\left(S[\D]\right)\cong\left(\mathrm{IndT}^{\C}_{\s,\D}(S)\right)[\C].\]
\end{enumerate}
\end{theorem}
There is  also a relatively different and recent interest for Turull’s and Puig’s induction. Note that in  \cite{art:Vit2022} Vitanov generalizes Turull’s and Puig’s induction to the case of sheaves of $G$-algebras and $\mathbb{C}G$-interior algebras (here $G$ is any group), in the context of $G$-equivariant topologies.

\section{Skew category algebras, twisted tensor product algebras and proof of Theorem \ref{thm:11}} \label{sec2}

Let $R:\mathcal{C}\to k$-$\Alg$ be a precosheaf of $k$-algebras on $\C.$ 

\goodbreak
\subsection{Skew category algebras.}\label{subsec:21} 
\medskip
We follow  \cite[Definition 3.3]{art:Wu2022} to recall that the skew category algebra  on $\mathcal{C}$ with respect to $R$, denoted  $R[\mathcal{C}]$, is the $k$-vector space spanned by the set
\[\{rf\mid f\in\mathrm{Mor}\,\mathcal{C}, r\in R(\mathrm{c}(f))\}.\]

The multiplication is given by the rule:
\begin{equation}
sg\ast rf=\left\{
\begin{array}{ll}
sR(g)(r)g\circ f,&\text{if } \mathrm{d}(g)=\mathrm{c}(f),\\
0, & \text{otherwise.}
\end{array}
\right.
\end{equation}
Extending the product linearly to any two elements, $R[\mathcal{C}]$ becomes an associative $k$-algebra. We have that $\obc$ is finite, then $\sum_{x\in \obc}1_{R(x)}1_x$ is the identity of $R[\mathcal{C}],$ hence $R[\mathcal{C}]$ is a $k$-algebra.

\subsection{A twisted tensor product on category algebras.}
\medskip

For shortness we will use in this subsection the following notations: by $A$ we denote the $k$-algebra
$\bigotimes_{x\in \mathrm{Ob}\,\mathcal{C}}R(x)$  and by $B$  we denote the category algebra $k\mathcal{C}.$
Let $\tau:B\otimes A\to A\otimes B$ be the map given by
\begin{equation} \label{eq:def_2}
\tau\left(f\otimes\left(\bigotimes_{x\in \mathrm{Ob}\,\mathcal{C}}r_x\right)\right)=
\left(\bigotimes_{t\in \mathrm{Ob}\,\mathcal{C}}r'_t\right)\otimes f,
\end{equation}
for any $f\in\mathrm{Mor}\,\mathcal{C}$, $r_x\in R(x)$ and $x\in\mathrm{Ob}\,\mathcal{C},$ where
\goodbreak
\[r'_t=\left\{\begin{array}{ll}
R(f)(r_{\mathrm{d}(f)}),& \text{if } t=\mathrm{c}(f)\\
r_x,&\text{if } t\neq \mathrm{c}(f)\  \text{and } t=x
\end{array}\right.\]

The multiplication $\mu_A$ in $A$  is the point wise multiplication induced by the multiplications of $R(x)$, $x\in \mathrm{Ob}\,\mathcal{C}$ and   $\mu_B$ is  the multiplication in the category algebra.

\goodbreak
\begin{proposition} \label{prop:2_1}
Since $\mathrm{Ob}\,\mathcal{C}$ is finite, we have that $\tau:B\otimes A\to A\otimes B$ is a twisting map.
\end{proposition}
\begin{proof}
Obviously, by our definition (\ref{eq:def_2}), $\tau$ is a $k$-linear map.

\smallskip
Next, we verify the three conditions of \cite[Proposition/Definition 2.3]{art:Cap1995}. Let $f\in\mathrm{Mor}\,\mathcal{C}$ and $\bigotimes_{x\in \mathrm{Ob}\,\mathcal{C}}r_x\in \bigotimes_{x\in \mathrm{Ob}\,\mathcal{C}}R(x)$. We obtain:
\[
\tau\left(f\otimes\left(\bigotimes\limits_{x\in \mathrm{Ob}\,\mathcal{C}}1_{R(x)}\right)\right) \quad \stackrel{(\ref{eq:def_2})}{=} \quad \left(\bigotimes\limits_{t\in \mathrm{Ob}\,\mathcal{C}}r'_t\right)\otimes f
\quad = \quad \left(\bigotimes\limits_{x\in \mathrm{Ob}\,\mathcal{C}}1_{R(x)}\right)\otimes f,
\]
since \[r'_t=\left\{\begin{array}{ll}
R(f)(1_{R_{\mathrm{d}(f)}}),& \text{if } t=\mathrm{c}(f)\\[0.1cm]
1_{R(x)},&\text{if } t\neq \mathrm{c}(f)\  \text{and } t=x
\end{array}\right.\]  and $R(f)$ is a unital homomorphism of $k$-algebras.
\[
\begin{array}{rcl}
\tau\left(\left(\sum\limits_{y\in \mathrm{Ob}\,\mathcal{C}}1_{y}\right)\otimes \left(\bigotimes\limits_{x\in \mathrm{Ob}\,\mathcal{C}}r_x\right)\right) & = & \sum\limits_{y\in \mathrm{Ob}\,\mathcal{C}}\tau\left(1_y\otimes\left(\bigotimes\limits_{x\in \mathrm{Ob}\,\mathcal{C}}r_x\right)\right)\\[0.4cm]
& = & \sum\limits_{y\in \mathrm{Ob}\,\mathcal{C}} \left(\bigotimes\limits_{t\in \mathrm{Ob}\,\mathcal{C}}r'_{t,y}\right)\otimes 1_y\\[0.4cm]
& = & \left(\bigotimes\limits_{x\in \mathrm{Ob}\,\mathcal{C}}r_x\right)\otimes \left(\sum\limits_{y\in \mathrm{Ob}\,\mathcal{C}} 1_y\right),
\end{array}
\]
where the last equality holds, since for any $y\in \obc$ we have  \[r'_{t,y}=\left\{\begin{array}{ll}
R(1_y)(r_y),& \text{if } t=y\\[0.1cm]
r_x,&\text{if } t\neq y\  \text{and } t=x.
\end{array}\right.\]

\smallskip
Next, we show that:
\[\tau\circ(\mu_B\otimes\mu_A)=(\mu_A\otimes\mu_B)\circ (\mathrm{id}_A\otimes \tau\otimes \mathrm{id}_B)\circ (\tau\otimes\tau)\circ(\mathrm{id}_B\otimes\tau\otimes\mathrm{id}_A).\]
Let $f,g\in\mathrm{Mor}\,\mathcal{C}$ and $\bigotimes_{x\in\mathrm{Ob}\,\mathcal{C}}s_x\in \bigotimes_{x\in\mathrm{Ob}\,\mathcal{C}}R(x)$. We have:
\[
\begin{array}{rcl}
I&:=&(\tau\circ(\mu_B\otimes\mu_A))\left((g\otimes f)\otimes\left(\left(\bigotimes\limits_{x\in\mathrm{Ob}\,\mathcal{C}}s_x\right)\otimes\left(\bigotimes\limits_{x\in\mathrm{Ob}\,\mathcal{C}}r_x\right)\right)\right)\\[0.4cm]
&=&\left\{
\begin{array}{ll}
\tau\left((g\circ f)\otimes\left(\bigotimes\limits_{x\in\mathrm{Ob}\,\mathcal{C}}s_xr_x\right)\right),&\text{if }\mathrm{d}(g)=\mathrm{c}(f)\\
0,&\text{otherwise}
\end{array}
\right.\\[0.7cm]
&=&\left\{
\begin{array}{ll}
\left(\bigotimes\limits_{t\in \mathrm{Ob}\,\mathcal{C}}p'_t\right)\otimes g\circ f,&\text{if }\mathrm{d}(g)=\mathrm{c}(f)\\
0,&\text{otherwise,}
\end{array}
\right.
\end{array}
\]
where, for   $x\in\obc,$ we have that
\[p'_t=\left\{\begin{array}{ll}
R(g\circ f)(s_{\mathrm{d}(g\circ f)}r_{\mathrm{d}(g\circ f)}),& \text{if } t=\mathrm{c}(g\circ f)\\[0.1cm]
s_xr_x,&\text{if } t\neq \mathrm{c}(g\circ f)\  \text{and } t=x.
\end{array}\right.\]
On the other hand
\begingroup
\allowdisplaybreaks
\[
\begin{array}{ll}
J&:=(\mu_A\otimes\mu_B)\\&\left(\id_A\otimes\tau\otimes\id_B\left(\tau\otimes\tau\left(\id_B\otimes\tau\otimes\id_A\left( (g\otimes f)\otimes
\left(\left(\bigotimes\limits_{x\in\mathrm{Ob}\,\mathcal{C}}s_x\right)\otimes\left(\bigotimes\limits_{x\in\mathrm{Ob}\,\mathcal{C}}r_x\right)\right) \right)\right)\right)\right)\\[0.5cm]
&= ((\mu_A\otimes\mu_B)\circ (\id_A\otimes\tau\otimes\id_B))
\left(\tau\otimes \tau
	\left(	\left(g				
			
			\otimes
			\left(
				\bigotimes\limits_{t\in\obc}s'_t
			\right)\right)
		\otimes \left( f
	
	\otimes
			\left(
				\bigotimes\limits_{x\in\obc}r_x
			\right)\right)\right)
		\right)\\[0.5cm]
& = ((\mu_A\otimes\mu_B)\circ (\id_A\otimes\tau\otimes\id_B))
	\left(	\left(\left(
				\bigotimes\limits_{t\in\obc}s''_t
			\right)
			\otimes
			g\right)
		\otimes \left( \left(
				\bigotimes\limits_{t\in\obc}r'_t
			\right)
	
	\otimes
			f\right)\right)\\[0.5cm]
			&= ((\mu_A\otimes\mu_B))
	\left(	\left(
				\bigotimes\limits_{t\in\obc}s''_t
			\right)
			\otimes
			\left(
				\bigotimes\limits_{t\in\obc}r''_t
			\right)\otimes g
		\otimes
			f\right)\\[0.5cm]

&= \left\{\begin{array}{ll}\left(
				\bigotimes\limits_{t\in\obc}s''_tr''_t
			\right)
			\otimes
			g\circ f, &\text{if } \mathrm{d}(g)=\mathrm{c}(f)\\[0.2cm]
			0,  &\text{otherwise,}
		\end{array}\right.
\end{array}
\]
\endgroup
where in the above  equality, for $x\in\obc,$ we have
\[
\begin{array}{rcl}
s'_t&=&\left\{\begin{array}{ll}
R(f)(s_{\mathrm{d}(f)}),& \text{if } t=\mathrm{c}(f)\\
s_x,&\text{if } t\neq \mathrm{c}(f)\  \text{and } t=x,
\end{array}\right.\\[0.4cm]

s''_t&=&\left\{\begin{array}{lll}
R(g)(R(f)(s_{\mathrm{d}(f)})),& \text{if } t=\mathrm{c}(g) \text{ and } \mathrm{d}(g)=\mathrm{c}(f)\\
R(g)(s_{\mathrm{d}(g)}),& \text{if } t=\mathrm{c}(g) \text{ and } \mathrm{c}(f)\neq \mathrm{d}(g)\\
s_x,&\text{if } t\neq \mathrm{c}(f)\  \text{and } t=x,
\end{array}\right.\\[0.7cm]

r'_t&=&\left\{\begin{array}{ll}
R(f)(r_{\mathrm{d}(f)}),& \text{if } t=\mathrm{c}(f)\\
r_x,&\text{if } t\neq \mathrm{c}(f)\  \text{and } t=x,
\end{array}\right.\\[0.4cm]

r''_t&=&\left\{\begin{array}{lll}
R(g)(R(f)(r_{\mathrm{d}(f)})),& \text{if } t=\mathrm{c}(g) \text{ and } \mathrm{d}(g)=\mathrm{c}(f)\\
R(g)(r_{\mathrm{d}(g)}),& \text{if } t=\mathrm{c}(g) \text{ and } \mathrm{c}(f)\neq \mathrm{d}(g)\\
r_x,&\text{if } t\neq \mathrm{c}(f)\  \text{and } t=x.
\end{array}\right.
\end{array}
\]
Since $\c(g)=\c(g\circ f)$ and $R(g\circ f)=R(g)\circ R(f)$ is a homomorphism of $k$-algebras, whenever $f,g \in \Mor\C$ and $\mathrm{d}(g)=\mathrm{c}(f),$ it follows that $I=J.$
\end{proof}
\goodbreak

\begin{corollary}\label{cor:2.2}
With the notations of Proposition \ref{prop:2_1}, $(A\otimes_{\tau}B,\mu_{\tau})$ is a twisted tensor product algebra.
More precisely,
$\left(\left(\bigotimes_{x\in\obc}R(x)\right)
\otimes k\C,\mu_{\tau}\right)$ is a twisted tensor product algebra, where $\mu_{\tau}=(\mu_A\otimes\mu_B)\circ (\id_A\otimes\tau\otimes\id_B).$
\end{corollary}

For the rest of this section we shall use the notation ``$\cdot$'' for above the product $\mu_{\tau},$ that is:
\[
\begin{array}{ll}
& \left(\left(\bigotimes\limits_{x\in\obc}s_x\right)\otimes g\right)\cdot \left(\left(\bigotimes\limits_{x\in\obc}r_x\right)\otimes f\right)\\[0.4cm]
=
& \left\{
	\begin{array}{ll}
		\left(\bigotimes\limits_{y\in\obc}p_y\right)\otimes (g\circ f),& \text{if }\mathrm{d}(g)=\mathrm{c}(f)\\
	0,&\text{otherwise,}
	\end{array}
\right.
\end{array}\]
where, for $y\in\obc,$ we have
\begin{equation}\label{eqmult_twisted}
\begin{array}{ll}
p_y=& \left\{
	\begin{array}{ll}
	s_{\c(g)}R(g)(r_{\d(g)}),& \text{if }y=\mathrm{c}(g)\\[0.2cm]
	\bigotimes\limits_{y\in\obc}s_yr_y,&\text{if }y\neq \mathrm{c}(g).
	\end{array}
\right.
\end{array}
\end{equation}

\goodbreak
\begin{proof}[Proof of Theorem \ref{thm:11}]
For any $f\in\Mor\C$, $r\in R(\c(f))$, we define
\[
\Psi(rf) = \left(\bigotimes\limits_{x\in\obc}r_x\right)\otimes f,
\]
where, for $x\in \obc $, we have \[
\begin{array}{ll}
r_x=& \left\{
	\begin{array}{ll}
	r,& \text{if }x=\mathrm{c}(f)\\
	1_{R(x)},&\text{if }x\neq \mathrm{c}(f).
	\end{array}
\right.
\end{array}
\]
The above map $\Psi$ extends to a $k$-linear homomorphism. We will show next that $\Psi$ is a homomorphism of $k$-algebras. Let $f,g\in\Mor\C$ and $s\in R(\c(g))$, $r\in R(\c(f))$. On the left hand side, we obtain
\[
\begin{array}{l}
	\Psi((sg)\ast(rf))\\[0.2cm]
	 =
	\left\{
		\begin{array}{ll}
			\Psi(sR(g)(r)g\circ f),& \text{if }\mathrm{d}(g)=\mathrm{c}(f)\\
			0,&\text{otherwise}
		\end{array}
	\right.\\[0.4cm]
	 =
	 \left\{
		\begin{array}{ll}
			\left(
				\bigotimes\limits_{x\in\obc}m_x		
			\right)\otimes(g\circ f)			
			,& \text{if }\mathrm{d}(g)=\mathrm{c}(f)\\
			0,&\text{otherwise,}
		\end{array}
	\right.
\end{array}
\]
where, for $x\in \obc,$ we have
\[\begin{array}{ll}
m_x=& \left\{
	\begin{array}{ll}
	sR(g)(r),& \text{if }x=\c(g)\\
	1_{R(x)},&\text{if }x\neq\c(g).
	\end{array}
\right.
\end{array}
\]
On the right hand side
\[
	\begin{array}{ll}
		\Psi(sg) \cdot \Psi(rf) \\[0.2cm]
		 =
				\left(\left(\bigotimes\limits_{x\in\obc}s_x\right)\otimes g\right)\cdot \left(\left(\bigotimes\limits_{x\in\obc}r_x\right)\otimes f\right)\\[0.4cm]
	 =
	 \left\{
		\begin{array}{ll}
			\left(
				\bigotimes\limits_{x\in\obc}m_x		
			\right)\otimes(g\circ f)			
			,& \text{if }\mathrm{d}(g)=\mathrm{c}(f)\\
			0,&\text{otherwise,}
		\end{array}
	\right.
\end{array}
\]
by applying (\ref{eqmult_twisted}) and, since \[
\begin{array}{ll}
s_x=& \left\{
	\begin{array}{ll}
	s,& \text{if }x=\mathrm{c}(g)\\
	1_{R(x)},&\text{if }x\neq \mathrm{c}(g)
	\end{array}
\right.
\end{array}
\text{ and }
\begin{array}{ll}
r_x=& \left\{
	\begin{array}{ll}
	r,& \text{if }x=\mathrm{c}(f)\\
	1_{R(x)},&\text{if }x\neq \mathrm{c}(f).
	\end{array}
\right.
\end{array}
\]
The following $k$-linear map
\[
\Phi:\left(\bigotimes_{x\in\obc}R(x)\right)\otimes_{\tau}k\C\to R[\C],\qquad
\Phi\left(\left(\bigotimes\limits_{x\in\obc}r_x\right)\otimes f\right) = r_{\c(f)}f,
\]
for any $f\in\Mor\C$, $r_x\in R(x)$, $x\in\obc$, is clearly a left inverse of $\Psi$.

Moreover $\Psi$ preserves the identities , since
\[
	\begin{array}{rcl}		
		\Psi\left(\sum\limits_{x\in\obc}1_{R(x)}1_x\right) & = & \sum\limits_{x\in\obc}\Psi(1_{R(x)}1_x)\\[0.2cm]
		& = & \sum\limits_{x\in\obc}\left(\bigotimes\limits_{t\in\obc}u_t\right)\otimes 1_x\\[0.4cm]
		& = & \left(\bigotimes\limits_{x\in\obc}1_{R(x)}\right)\otimes 1_{k\C},
	\end{array}
\]
since, for some fixed $x\in\obc,$ we have
\[\begin{array}{ll}
u_t=& \left\{
	\begin{array}{ll}
	1_{R(x)},& \text{if }t=x\\
	1_{R(t)},& \text{if }t\neq x,
	\end{array}
\right.
\end{array}
\]
for any $t\in\obc.$
\end{proof}

\pagebreak
\section{Surjective Turull's induction of precosheaves of \texorpdfstring{$k$}{k}-algebras} \label{sec3}

Let $\C, \D$ be  two finite categories.
\color{black}

\subsection{Precosheaves of \texorpdfstring{$k$}{k}-algebras and modules over category algebras.} \label{subsec:31}
\medskip

In this subsection let $S:\D\to k\textrm{-}\Alg$ be a precosheaf of $k$-algebras on $\D.$  We identify $k\D\textrm{-}\mathrm{mod}$ and $(k\textrm{-}\mathrm{Vect})^{\D}$ (if this does not cause any trouble), in the following manner (see \cite[Section 2.2]{art:Xu2007}).
Since $S\in (k\textrm{-}\mathrm{Vect})^{\D}$, we define the $k\D$-module $M_S$ to be
\[M_S=\bigoplus_{x\in\obd}S(x),\]
with the left $k\D$-action given by
\[
f\cdot m=
\begin{cases}
S(f)(m),& \text{if }x=\d(f)\\
0,&\text{otherwise},
\end{cases}
\]
for any $x\in\obd$, $m\in S(x)$, $f\in\Mor\D$. Conversely, if $M\in k\D\textrm{-}\mathrm{mod}$, we define the functor $S_M\in (k\textrm{-}\mathrm{Vect})^{\D}$ such that $$S_M(x)=1_x\cdot M,$$ for any $x\in\obd$. If $f\in\Hom_{\D}(x,y)$, then $S_M(f)\in\Hom_{k\textrm{-}\mathrm{Vect}}(S_M(x),S_M(y))$ is given by $S_M(f)(m)=f\cdot m$, for any $x,y\in\obd$, $m\in 1_x\cdot M$.

\medskip
\subsection{Surjective Turull's induction of precosheaves of \texorpdfstring{$k$}{k}-algebras.}
\medskip

Let $\s:\D\to\C$ be a covariant functor. We  need the following (probably) well-known relation on the set of morphisms  $\Mor\D$:
\begin{equation}\label{eq4}
 f_1\sims f_2 \Longleftrightarrow \s(f_1)=\s(f_2),
\end{equation}
for any $f_1,f_2\in\Mor\D.$

The following (folklore) lemmas are easy to verify and contain statements about $\s$ and about the relation $\sims$, defined above.
\begin{lemma} \label{lemma:343}
Let $f_1,f_2,g_1,g_2\in\Mor\D$. The following statements are true:
\begin{enumerate}[\normalfont (i)]
\item $\sims$ is an equivalence relation on $\Mor\D$;
\item If $f_1 \sims f_2$, $g_1\sims g_2$, $\d(g_1)=\c(f_1)$ and $\d(g_2)=\c(f_2)$ then $g_1\circ f_1\sims g_2\circ f_2$;
\item If $\s$ is injective on objects and $f_1\sims f_2$ then $\d(f_1)=\d(f_2)$ and $\c(f_1)=\c(f_2)$.
\end{enumerate}
\end{lemma}

\begin{lemma} \label{assumptions:341}
If $\s:\D\to\C$ is an  injective on objects, surjective on morphisms covariant functor, then:
\begin{itemize}
\item[(i)] $\s(\D)$ is a subcategory of $\C$;
\item[(ii)] $\s$ is surjective on objects; in particular it is bijective on objects.
\end{itemize}
\end{lemma}
Throughout this subsection let $S:\D\rightarrow k\textrm{-}\Alg$ be a precosheaf of $k$-algebras on $\D.$

\begin{definition} \label{def:344}
Let $\s:\D\to\C$ be an  injective on objects, surjective on morphisms covariant functor and let $x\in\obd$. We consider the  $k$-subspace $M_S(x)$ of $S(x),$ that is formed by all elements
$m_x\in S(x),$ such that $$S(f)(m_x)=S(g)(m_x),$$ for any $f\in\Mor\D$ with $\d(f)=x$ and any $ g\in\sims<f>;$ further, we obtain the $k$-space
\[\oM=\bigoplus_{x\in\obd}M_S(x).\]
We also define the next operation \[h\cdot m_x=\begin{cases}
S(f)(m_x),&\text{if }f\in\Mor\D\text{ such that } \s(f)=h \text{ and } x=\d(f),\\
0,&\text{otherwise,}
\end{cases}\]
for any $h\in\Mor\C$, $m_x\in M_S(x)$.

\end{definition}
\goodbreak

\begin{proposition} \label{prop:345} Let $\s:\D\to\C$ be an  injective on objects, surjective on morphisms covariant functor and let $x\in\obd$. Then:
\begin{enumerate}[{\normalfont (i)}]
\item $M_S(x)$ is a $k$-subalgebra of $S(x)$;
\item With component-wise multiplication $\oM$  is a $k$-subalgebra of  $\prod_{x\in\obd}S(x);$
\item $\oM$ is a left $k\C$-module, with the above operation ``$\cdot$''.
\end{enumerate}
\end{proposition}

\begin{proof}
(i) Clearly $M_S(x)$ is a $k$-submodule of $S(x)$.

We first verify that $1_{S(x)}\in M_S(x),$ which is true if and only if $S(f)(1_{S(x)})=S(g)(1_{S(x)})$, for all $f\in\Mor\D$, with $\d(f)=x$ and for all $g\in\sims<f>$. Let $f\in\Mor\D$ such that $\d(f)=x$ and $g\in\Mor\D$ such that $\s(f)=\s(g)$. By Lemma \ref{lemma:343}.(iii) it follows that $\d(f)=\d(g)=x$ and $\c(f)=\c(g)$ hence
\[S(f)(1_{S(x)})=1_{S(\c(f))}=1_{S(\c(g))}=S(g)(1_{S(x)})\]

Next, let $m_x,m'_x\in M_S(x)$ and $f\in\Mor\D$ such that $\d(f)=x$ and $g\in\sims<f>$. We obtain:
\[S(f)(m_x m'_x)=S(f)(m_x)\cdot S(f)(m'_x)=S(g)(m_x)\cdot S(g)(m'_x)=S(g)(m_x m'_x).\]

\smallskip
(ii) This  is a consequence of (i).

\smallskip
(iii) First we justify why ``$\cdot$'' is well-defined: assume that there exist $f_1,f_2\in\Mor\D$ such that 
\[\s(f_1)=h,\qquad \s(f_2)=h,\qquad x=\d(f_1)=\d(f_2).\] 
It follows that $\s(f_1)=\s(f_2)$ and $\c(f_1)=\c(f_2),$ hence $f_1\in\sims<f_2>$. Using Definition \ref{def:344} we obtain $$S(f_1)(m_x)=S(f_2)(m_x),$$ since $m_x\in M_S(x).$\qedhere
\end{proof}

In the following remark we describe explicitly how does $k\C$ acts on $\oM.$
\begin{remark} \label{remark:345p} We assume that $\s:\D\to\C$ is an  injective on objects, surjective on morphisms covariant functor.
\begin{itemize}
\item[(i)]
Let $\sum_{g\in\Mor\C}\alpha_g g\in k\C$ and $\bigoplus_{x\in\obd}m_x\in\oM,$ where $m_x\in M_S(x),x\in \obd.$ Then 
\[
\left(\sum\limits_{g\in\Mor\C}\alpha_g g\right)\cdot\left(\bigoplus\limits_{x\in\obd}m_x\right)  =  \sum\limits_{g\in\Mor\C}\left(g\cdot\bigoplus\limits_{x\in\obd}m_x \right).
\]
But, for each $g\in\Mor\C$ there is a unique equivalence class $\sims<f>$, with $f\in\Mor\D$ such that  $\s(f)=g.$ It follows that $g$ acts on the elements $m_x\in M_S(x)$ such that $x=\d(f)$, hence
\[
\left(\sum\limits_{g\in\Mor\C}\alpha_g g\right)\cdot\left(\bigoplus\limits_{x\in\obd}m_x\right) =  \sum\limits_{g\in\Mor\C}\bigoplus\limits_{f\in\Mor\D ,\s(f)=g}g\cdot m_{\d(f)}.
\]
\item[(ii)] If $1_{x'}\in\Mor\C$, $x'\in\obc$, then there exists a unique $x\in\obd$ such that $\s(x)=x'$, hence $\s(1_x)=1_{\s(x)}=1_{x'}$. Therefore
\[
1_{x'}\cdot\left(\bigoplus_{x\in\obd}m_x\right)=S(1_x)(m_x)=\id_{S(x)}(m_x)=m_x.
\]
Thus
\[
\left(\sum_{x'\in\obc}1_{x'}\right)\cdot \bigoplus_{x\in\obd}m_x = \bigoplus_{x\in\obd}m_x,
\]
since $\s$ is bijective on objects, by Lemma \ref{assumptions:341}, (ii).
\end{itemize}
\end{remark}
\color{black}
\begin{definition}\label{def:37}
Let $\s:\D\to\C$ be an  injective on objects, surjective on morphisms covariant functor. Using the discussion of Subsection \ref{subsec:31} and Proposition \ref{prop:345}.(iii) we obtain the following functor, the \textbf{surjective Turull's induction of precosheaves of $k$-algebras}:
\[
\Ind^{\C}_{\s,\D}(S):\C\to k\m\mathrm{Vect}\qquad \Ind^{\C}_{\s,\D}(S)(x')=1_{x'}\cdot\oM,
\]
for all $x'\in\obc.$ For any $g\in\Hom_{\C}(x',y')$, where $x',y'\in\obc,$ we have
\[
\Ind^{\C}_{\s,\D}(S)(g):1_{x'}\cdot \oM\to 1_{y'}\cdot\oM\qquad \Ind^{\C}_{\s,\D}(S)(g)(m)=g\cdot m,
\]
for any $m\in\oM.$
\end{definition}
\goodbreak

\begin{proposition}\label{prop:38}
Let $S:\D\to k \textrm{-}\Alg$ be a precosheaf of $k$-algebras. Let $\s:\D\to\C$ be an  injective on objects, surjective on morphisms covariant functor. Then $\Ind^{\C}_{\s,\D}(S)$ is a precosheaf of $k$-algebras.
\end{proposition}

\begin{proof}
Let $x',y'\in\obc$, $g\in\Hom_{\C}(x',y')$ and let $x,y\in\obd$ be the unique objects of $\C$ such that $\s(x)=x'$ and $\s(y)=y'$.
\[
\begin{array}{rcl}
\Ind^{\C}_{\s,\D}(S)(x') & = & 1_{x'}\cdot \oM\\[0.1cm]
& = & 1_{x'}\cdot\left(\bigoplus\limits_{x\in\obd}M_S(x)\right)\\[0.4cm]
& = & S(1_x)(M_S(x))\\[0.1cm]
& = & \id_{S(x)}(M_S(x))\\[0.1cm]
& = & M_S(x),
\end{array}
\]
which is a unitary $k$-algebra (see  Remark \ref{remark:345p}.(ii) and Proposition \ref{prop:345}).
Let $f\in\Hom_{\D}(x,y)$ such that $\s(f)=g$, then $\Ind^{\C}_{\s,\D}(S)$ is given by:
\[
\begin{array}{l}
\Ind^{\C}_{\s,\D}(S)(g):M_S(x)\to M_S(y),\quad
\Ind^{\C}_{\s,\D}(S)(g)(m_x)=S(f)(m_x),
\end{array}
\]
for all $m_x\in S(x);$
which is a homomorphism of $k$-algebras.
\end{proof}

\section{Surjective Puig's  induction of interior category algebras and proof of Theorem \ref{thm:13}}\label{sec4}

Let $\C$, $\D$ be finite  categories and let $\s:\D\to\C$ be an injective on objects, surjective on morphisms covariant functor. Then, by Lemma \ref{assumptions:341}.(ii) we know that $\s:\obd\to\obc$ is a bijection given by $\s(x)=x'$, for any $x\in\obd$. For $x\in\obd$ with $\s(1_x)=1_{x'}$ (i.e. $\s(x)=x'$), we denote by $\mathrm{Id}_x$ the equivalence class $\sims<1_x>$, that is
\[\mathrm{Id}_x=\left\{f\in\Mor\D\mid \s(f)=1_{x'}\right\}.\]
We consider the following subset of $k\D$
\[
\S=\left\{\sum_{x\in\obd}f_x\mid f_x\in\mathrm{Id}_x, x\in\obd \right\}.
\]

\begin{lemma} \label{lemma:421}
Let $x\in\obd$.
\begin{enumerate}[\normalfont  (i)]
\item If $\obd=\{x_1,\ldots,x_n\}$, $s_1=f_1+\ldots+f_n$, $s_2=g_1+\ldots+g_n$ with $f_i,g_i\in\mathrm{Id}_x$, $i\in\{1,\ldots,n\}$ then \[s_1\cdot s_2=f_1\circ g_1+\ldots+f_n\circ g_n\in\S;\]
\item $(\S,\cdot)$ is a submonoid of $(k\D,\cdot)$;
\item $k\S$ is a $k$-subalgebra of $k\D$ and as a monoid algebra has an augmentation map, which we denote by $\varepsilon:k\S\to k$, $\varepsilon(s)=1_{k}$, for any $s\in\S$.
\end{enumerate}
\end{lemma}

\goodbreak
\begin{lemma} \label{lemma:422} We assume that
\begin{equation}\label{eq:423}
\sims<f>=\mathrm{Id}_y\circ f=f\circ\mathrm{Id}_x, \quad \text{for any } f\in\Hom_{\D}(x,y).
\end{equation}
\begin{enumerate}[\normalfont (i)]
\item If $f\in\Mor\D$, $s\in\S$ then there exist $s',s''\in\S$ such that \[s\cdot f=f\cdot s',\quad f\cdot s=s''\cdot f;\]
\item Let $S:\D\to k\mathrm{-}\Alg$ be a precosheaf of $k$-algebras on $\D$ and $x\in\obd$. The following statements are equivalent:
\begin{enumerate}[\normalfont (a)]
\item $m_x\in M_S(x)$;
\item $S(f_x)(m_x)=m_x$, for any $f_x\in\mathrm{Id}_x$.
\end{enumerate}
\end{enumerate}
\end{lemma}

\begin{proof}
(i) Let $f\in\Hom_{\D}(x,y)$ and $s\in\S$. Then $s=f_1+f_2+\ldots+f_n$, with $f_1\in\mathrm{Id}_{x_1}, f_2\in\Id_{x_2},\ldots,f_n\in\Id_{x_n}$, where $\obd=\{x_1,x_2,\ldots,x_n\}$. It follows that there is a unique $i_0\in\{1,2,\ldots,n\}$ such that $x_{i_0}=y$ and
\[s\cdot f=(f_1+\ldots+f_n)\cdot f=f_{i_0}\circ f\in\Id_y\circ f,\]
hence by (\ref{eq:423}) we obtain
\[s\cdot f=f_{i_0}\circ f=f\circ f'_{i_1},\quad i_1\in\{1,2,\ldots,n\},\]
where $f'_{i_1}\in\Id_x$ and $x=x_{i_1}.$ We consider
\[s':=1_{x_1}+\ldots+f'_{i_1}+\ldots+1_{x_n} \in\S \]
hence $f\cdot s'=f\circ f'_{i_1};$ similarly  $f\cdot s=s''\cdot f$.

\medskip
(ii) Recall that (a) is equivalent with the following statement: if $f_0\in\Mor\D$ such that $\d(f)=x$ then $S(f_0)(m_x)=S(f_1)(m_x)$ for any $f_1\in\sims<f_0>.$

\smallskip
``(a)$\Rightarrow$(b)'' Let $f_x\in\Id_x.$  It follows that $S(f_x)(m_x)=S(1_x)(m_x)$.

\smallskip
``(b)$\Rightarrow$(a)'' Let $f_0\in\Mor\D$ such that $\d(f_0)=x$ and $f_1\in\sims<f_0>$. By (\ref{eq:423}) we get that $f_1=f_0\circ f'_x$, where $f'_x\in\Id_x$. Then
\[S(f_1)(m_x) = S(f_0\circ f'_x)(m_x)=S(f_0)\left(S(f'_x)(m_x)\right)\overset{\text{(b)}}{=} S(f_0)(m_x).\]
\end{proof}
In the following example we present situations of functors and finite categories for which condition (\ref{eq:423}) is satisfied. The example described in statement b) assure us that there are cases of category algebras and functors for which Theorem \ref{thm:13} can be applied but we cannot apply the results of  \cite{art:CMT2017}.
\begin{example}
\begin{enumerate}[a)]
\item Let $\H$, $\G$ be two groupoids and $\s:\H\to\G$ be any injective on objects, surjective on morphisms covariant functor. Then (\ref{eq:423}) is true. If we consider posets viewed as categories then any injective  on objects, surjective on morphisms covariant functor between two posets becomes an isomorphism of posets and (\ref{eq:423}) is trivially true.
\item Let $\D, \C$ be the following categories with two objects. We consider $\obd=\{x,y\}, \obc=\{x',y'\},$ the morphisms $$\Mor \D =\{1_x,1_y,f_1,f_2\},\quad \Mor \C=\{1_{x'},1_{y'},f_1'\},$$
and the compositions 
\[
\begin{array}{lll}
f_2\circ f_1=f_1\circ 1_x=1_y\circ f_1=f_1,&\hspace{1cm}& f_2\circ f_2=1_y,\\[0.2cm]
f_2\circ 1_y=1_y\circ f_2=f_2,&& f_1'\circ 1_{x'}=1_{y'}\circ f_1'=f_1'.
\end{array}
\]
The functor $\s:\D\rightarrow \C,$  represented in the following diagram
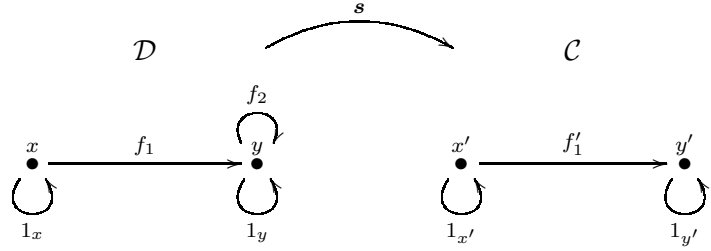
\[
\xymatrix@C+=1cm{
&\D  & \ar@/^/@/^1pc/[rr]^{\s} & & &  \C \\
 \stackrel{x}{\bullet} \ar[rr]^{f_1}\ar@(dl,dr)_{1_x} & & \stackrel{y}{\bullet}  \ar@(dl,dr)_{1_y} \ar@(ul,ur)^{f_2}  & & \stackrel{x'}{\bullet} \ar[rr]^{f_1'}\ar@(dl,dr)_{1_{x'}} & & \stackrel{y'}{\bullet}  \ar@(dl,dr)_{1_{y'}} \\ 
& & & & &
}
\] is given by $\s(x)=x',s(y)=y'$ and
$$\s(f_1)=f_1',\qquad\s(1_x)=1_{x'},\qquad \s(f_2)=\s(1_y)=1_{y'}.$$

It is easy to verify that $\s$ is an injective on objects, surjective on morphisms covariant functor satisfying condition (\ref{eq:423}). Let $\Delta:k\D\rightarrow k\D\otimes k\D$ be the natural candidate for a co-multiplication, $\Delta(\sum_{f\in\Mor \D}\lambda_f f)=\sum_{f\in\Mor \D}\lambda_f f\otimes f$ and the twist map $\tau:k\D\otimes k\D\rightarrow k\D\otimes k\D$ defined on basis elements $\tau(f\otimes g)=g\otimes f,$ for any $f,g\in \Mor \D.$
There is no Hopf algebra structure on the category algebra $k\D$ with co-multiplication map $\Delta$ and twist map $\tau.$ By contradiction, if there is a Hopf algebra on $k\D,$ since $k\D$ is of dimension four, we apply \cite[Theorem 3.5, Table I]{art:Ste1999} to obtain that $k\D$ is isomorphic with either $kC_4, k(C_2\times C_2)$ or $T_4$ (the unique noncommutative and noncocommutative Hopf algebra of dimension four); see also \cite[Corollary 3]{art:Ste1997}. Clearly $k\D$ is noncommutative and, with the above $\Delta $ and $\tau$, it is cocommutative, see \cite[Proposition 2.2]{art:Xu2013}. 

Note that the above example suggests us to search the answer for the following question: which are the conditions for a fixed, given algebra (or coalgebra) $A$ to have a Hopf algebra structure?  By email discussions with Professor Gigel Militaru we found out that it seems that there is no systematic approach for this question and it is not an easy question.
\end{enumerate}
\end{example}

If $(C,\sigma)$ is an interior $A$-algebra we denote by ${}_{\sigma}C_{\sigma}$ the $(A,A)$-bimodule structure induced by $\sigma.$

\begin{definition}\label{def:44}
Let $(C,\sigma)$ be a  interior $k\D$-algebra. We  still denote by $_{\sigma}C_{\sigma}$ the  $(k\S,k\S)$-bimodule structure of $C$ obtained  by restriction to $k\S$ of the $(k\D,k\D)$-bimodule structure. Using Lemma \ref{lemma:421}.(iii) it follows that $k$ is an $(k\S,k\S)$-bimodule (through $\varepsilon$) that is denoted $_{\varepsilon}k_{\varepsilon}$. Therefore, we can consider the $(k\S,k\S)$-bimodule $_{\varepsilon}k\otimes_{k\S}C_{\sigma}$. The \textbf{surjective Puig's induction} from $\D$ to $\C$ of $C$ is
\[
\IndP^{\C}_{\s,\D}(C)=\left(_{\varepsilon}k\otimes_{k\S}C_{\sigma}\right)^{\S}.
\]
\end{definition}

\begin{proposition}\label{prop:45}
Let $\s:\D\to\C$ be an injective on objects, surjective on morphisms covariant functor. Let $(C,\sigma)$ be a $k\D$-interior algebra.
\begin{enumerate}[\normalfont (i)]
\item Then $\IndP^{\C}_{\s,\D}(C)$ is a $k$-algebra;
\item If $\mathrm{(\ref{eq:423})}$ is true then $\IndP^{\C}_{\s,\D}(C)$ is a $k\C$-interior algebra.
\end{enumerate}
\end{proposition}

\begin{proof}
(i) We use the pointwise multiplication. Let $a_1\otimes c_1,a_2\otimes c_2\in\IndP^{\C}_{\s,\D}(C)$. For any $s\in\S$, we have:
\[
\begin{array}{rclcl}
\left((a_1\otimes c_1)\cdot (a_2\otimes c_2)\right)\cdot s & = & (a_1a_2\otimes c_1c_2)s & = & a_1a_2\otimes c_1c_2\sigma(s)\\
& = & (a_1\otimes c_1)(a_2\otimes c_2\sigma(s)) & = & (a_1\otimes c_1)(a_2\otimes c_2)
\end{array}
\]
The identity is $1_{\IndP^{\C}_{\s,\D}(C)}=1\otimes 1_C,$ since for any $s\in\S$:
\[
\begin{array}{rclcl}
(1\otimes 1_C) \cdot s & = & 1\otimes 1_C\cdot \sigma(s) & = & 1\otimes \sigma(s)1_C\\
& = & 1\otimes_{\S}\sigma(s)1_C & = & 1\cdot\varepsilon(s)\otimes 1_C\\
& = & 1\otimes 1_C.
\end{array}
\]

\goodbreak
\medskip
(ii) Let $x,y\in\obd$ and the corresponding objects $x',y'\in\obc.$  Let \[\tau :k\C\to \left(_{\varepsilon}k\otimes_{k\S}C_{\sigma}\right)^{\S},\quad \tau(f')=1\otimes\sigma(f),\] for every $f'\in\Hom_{\C}(x',y');$ where $f\in\Hom_{\D}(x,y)$ such that $\s(f)=f',$ since $\s$ is surjective on morphisms. We show first that $\tau$ is well-defined. For this, let $f_1,f_2\in\Hom_{\D}(x,y)$, such that \[\s(f_1)=\s(f_2)=f'.\] hence $f_2\in\sims<f_1>$. By (\ref{eq:423}) it follows that there are $g_y\in\Id_y$, $g_x\in\Id_x$ such that \[f_2=g_y\circ f_1=f_1\circ g_x.\] Let $s_0\in\S$, where $s_0$ is the sum of all identity homomorphisms in $\D$, excepting position $y$, where we replace $1_y$ by $g_y$ (that is, if $\obd=\{x_1,\ldots,x_n\}$, then $s_0=1_{x_1}+1_{x_2}+\ldots+g_y+\ldots+1_{x_n}$). It follows $f_2=g_0\circ f_1=s_0\circ f_1$ and
\[
\begin{array}{rclcl}
1\otimes_{k\S}\sigma(f_2) & = & 1\otimes_{k\S}\sigma(s_0\cdot f_1) & = & 1\otimes_{k\S}\sigma(s_0)\sigma(f_1)\\[0.2cm]
& = & \varepsilon(s_0)\otimes_{k\S}\sigma(f_1) & = & 1\otimes_{k\S}\sigma(f_1).
\end{array}
\]
Moreover, for any $s\in\S$, by Lemma \ref{lemma:422}.(ii), there is $s''\in\S$ such that $f\cdot s=s''\cdot f$. We obtain:
\[
\begin{array}{rclcl}
\left(1\otimes_{k\S}\sigma(f)\right)\cdot s & = & 1\otimes_{k\S}\sigma(f\cdot s) & = & 1\otimes_{k\S}\sigma(s''\cdot f)\\[0.2cm]
& = & \varepsilon(s'')\otimes_{k\S}\sigma(f) & = & 1\otimes_{k\S}\sigma(f).
\end{array}
\]

We show next that $\tau$ is a homomorphism of $k$-algebras. For every $f'_1,f'_2\in\Mor\C$, we have:
\[
\begin{array}{rcl}
\tau(f'_1\cdot f'_2) & = & \left\{
\begin{array}{ll}
\tau(f'_1\circ f'_2), & \text{if } \d(f'_1)=\c(f'_2)\\[0.1cm]
0, & \text{if }\d(f'_1)\neq\c(f'_2)
\end{array}
\right.\\[0.5cm]
 & = & \left\{
\begin{array}{ll}
1\otimes\sigma(f_1\circ f_2), & \text{if } \d(f_1)=\c(f_2)\\[0.1cm]
0, & \text{if }\d(f_1)\neq\c(f_2)
\end{array}
\right.\\[0.5cm]
& = & \tau(f'_1)\cdot\tau(f'_2),
\end{array}
\]
and
\[
\tau\left(\sum_{x'\in\obc}1_{x'}\right) = 1\otimes \sigma\left(\sum_{x\in\obd}1_{x}\right)=1\otimes_{k\S}1_C.
\]
\end{proof}

\begin{proof}[Proof of Theorem \ref{thm:13}]
(i) We have that $k\D$ is $\D$-graded with
\[k\D=\bigoplus_{f\in\Mor\D}A_f,\qquad A_f=kf\]
and $S[\D]$ is $\D$-graded with
\[S[\D]=\bigoplus_{f\in\Mor\D}B_f,\qquad B_f=S_{\c(f)}f.\]
Clearly $\sigma(A_f)\subseteq B_f$, for every $f\in\Mor\D$ and
\[\sigma\left(\sum_{x\in\obd}1_x\right)=\sum_{x\in\obd}\sigma(1_x)=\sum_{x\in\obd}1_{S(x)}1_x.\]
Let $f,g\in\Mor\D$. We have:
\[
\begin{array}{rcl}
\sigma(g\cdot f) & = &
\left\{
\begin{array}{ll}
\sigma(g\circ f), & \text{if }\d(g)=\c(f)\\[0.1cm]
0, & \text{if }\d(g)\neq\c(f)
\end{array}
\right.\\[0.5cm]
& = &
\left\{
\begin{array}{ll}
1_{R(\c(g\circ f))}g\circ f, & \text{if }\d(g)=\c(f)\\[0.1cm]
0, & \text{if }\d(g)\neq\c(f);
\end{array}
\right.
\end{array}
\]
\[
\begin{array}{rcl}
\sigma(g)\ast\sigma(f) & = & \left(1_{S(\c(g))}g\right)\ast \left(1_{S(\c(f))}f\right)\\[0.5cm]
& = &
\left\{
\begin{array}{ll}
1_{S(\c(g))}S(g)\left(1_{S(\c(f))}\right)g\circ f, & \text{if }\d(g)=\c(f)\\[0.1cm]
0, & \text{if }\d(g)\neq\c(f)
\end{array}
\right.\\[0.5cm]
& = &
\left\{
\begin{array}{ll}
1_{S(\c(g))}g\circ f, & \text{if }\d(g)=\c(f)\\[0.1cm]
0, & \text{if }\d(g)\neq\c(f).
\end{array}
\right.
\end{array}
\]

\medskip
(ii) We denote by
\[
\psi:\left(\mathrm{IndT}^{\C}_{\s,\D}(S)\right)[\C]\to\left(k\otimes_{k\S}S[\D]\right)^{\S},\qquad
\psi(m_{\c(f)}f')=1\otimes m_{\c(f)}f,
\]
for any $m_{\c(f)}f'\in \left(\mathrm{IndT}^{\C}_{\s,\D}(S)\right)[\C]$.

First, we verify that $\psi$ is well defined. Let $f_1,f_2\in\Mor\D$ such that $\s(f_1)=\s(f_2)=f'$. It follows that $f_2\in\sims<f_1>$ hence, by (\ref{eq:423}) we obtain that there is $g_{\c(f_1)}\in\Id_{\c(f_1)}$ with $f_2=g_{\c(f_1)}\circ f_1$ and:
\[
\begin{array}{rcl}
1\otimes m_{\c(f_1)}f_2 & = & 1 \otimes  m_{\c(f_1)}g_{\c(f_1)}\circ f_1\\[0.1cm]
& = & 1 \otimes \left(1_{S(\c(f_1))}g_{\c(f_1)}\right)\ast\left(m_{\c(f_1)}f_1\right)\\[0.1cm]
& = & 1 \otimes   \sigma(r)\cdot m_{\c(f_1)} f_1\\[0.1cm]
& =& \varepsilon(r) \otimes  m_{\c(f_1)} f_1\\[0.1cm]
& = & 1 \otimes m_{\c(f_1)} f_1,
\end{array}
\]
where the second equality is true since $m_{\c(f_1)}\in M_S(c(f_1))$ and $$r:=1_{x_1}+\ldots+g_{\c(f_1)}+\ldots+1_{x_n}$$ (here $\c(f_1)\in\{x_1,\ldots,x_n\}$ which is the set $\obd$).

Next, we show that $1\otimes m_{\c(f)}f\in\left(k\otimes_{k\S}S[\D]\right)^{\S}$. Let $s\in\S$, hence $s=\sum_{x\in\obd}g_x$, $g_x\in\Id_x$, $x\in\obd$. Then
\[
\begin{array}{rcl}
\left(1\otimes m_{\c(f)}f\right)\cdot s & = & 1\otimes_{}m_{\c(f)}f\cdot\sigma(s)\\[0.2cm]
& = & 1\otimes \sum\limits_{x\in\obd}(m_{\c(f)}f)\ast(1_{S(x)}g_x)\\[0.3cm]
& = & 1\otimes_{} \sum\limits_{x\in\obd}m_{\c(f)}S(f)(1_{S(x)})f\cdot g_x\\[0.3cm]
& = & 1\otimes_{} m_{\c(f)}(f\circ g_{\d(f)})\\[0.1cm]
& = & 1\otimes_{} m_{\c(f)}\left(  g'_{\c(f)}\circ f\right)\\[0.1cm]
& = & 1\otimes_{} \left(1_{S(\c(f))}g'_{\c(f)}\right)\ast \left(m_{\c(f)}f\right)\\[0.1cm]
& = & 1 \otimes \sigma(s'') \cdot m_{\c(f)}f\\[0.1cm]
& = & 1\otimes m_{\c(f)}f,
\end{array}
\]
where the fourth equality is true by (\ref{eq:423}) and since $x=\mathrm{d}(f);$ the fifth equality is true since $m_{\c(f)}\in M_S(\c(f))$ and 
$$s'':=1_{x_1}+\ldots+g'_{\c(f)}+\ldots+1_{x_n},  \qquad \c(f)\in\{x_1,\ldots,x_n\}\text{ if }\obd=\{x_1,\ldots,x_n\}.$$

Clearly $\psi$ is a homomorphism of $k$-algebras, since:
\[
\begin{array}{rcl}
\psi((m_{\c(g)}g')\ast (m_{\c(f)}f')) & = &
\left\{
\begin{array}{ll}
\psi\left(m_{\c(g)}\mathrm{IndT}^{\C}_{\s,\D}(S)(g')(m_{\c(f)})g'\circ f'\right), & \text{if }\d(g')=\c(f')\\[0.1cm]
0, & \text{if }\d(g')\neq\c(f')
\end{array}
\right.\\[0.5cm]
& = &
\left\{
\begin{array}{ll}
1\otimes m_{\c(g)}S(g)(m_{\c(f)})g\circ f, & \text{if }\d(g')=\c(f')\\[0.1cm]
0, & \text{if }\d(g')\neq\c(f'),
\end{array}
\right.
\end{array}
\]
and
\[
\begin{array}{rcl}
\psi(m_{\c(g)}g')\cdot \psi(m_{\c(f)}f') & = &
(1\otimes m_{\c(g)}g)(1\otimes m_{\c(f)}f)\\[0.2cm]
& = &
\left\{
\begin{array}{ll}
1\otimes m_{\c(g)}S(g)(m_{\c(f)})g\circ f, & \text{if }\d(g)=\c(f)\\[0.1cm]
0, & \text{if }\d(g)\neq\c(f).
\end{array}
\right.
\end{array}
\]
Since $\s$ is bijective on objects, we know that $\d(g)=\c(f)$ if and only if $\d(g')=\c(f')$.

 For $f'\in \Mor(\C)$ the homogeneous $f'$-component of
 $\mathrm{IndT}^{\C}_{\s,\D}(S)[\C]$ is $M_S(\mathrm{c}(f'))f',$ and the homogeneous $f'$-component $(k\otimes_{k\S}S[\D])^{\S}_{f'}$ of $(k\otimes_{k\S}S[\D])^{\S}$ is 
 $$\left(k\otimes_{k\S}\bigoplus_{g\in \sims<f> }S[\D]_g\right)^{\S}.$$
It is clear that $\psi $ is a homomorphism of $\C$-graded algebras.

We define the map \[\theta:k\otimes_{k\S}S[\D]\to M_S\otimes k\C, \quad \theta(1\otimes m_{\c(f)}f)=m_{\c(f)}\otimes \s(f),\] for all $f\in\Mor\D$ and $m_{\c(f)}\in S(\c(f)).$

 Clearly $\theta$ is a homomorphism of right $ k\S$-modules, hence restricts on $\S$-fixed submodules; we explain this in the next lines. The opposite category $\D^{\op}$ has the same objects as $\D,$ but $f^{\op}\in \Hom_{\D^{\op}}(y,x)$ if and only if $f\in \Hom_{\D}(x,y)$ and, the composition is given by $f^{\op}\circ g^{\op}=(g\circ f)^{\op},$ for any $x,y\in\obd, f,g\in \Mor\D$ with $\d(g)=\c(f).$ Now $ S:\D^{\op}\rightarrow k\textrm{-}\Alg,$ defined by  $S(f^{\op})=S(f),$ for any $f^{\op}\in\Mor \D^{\op}$ is a contravariant functor. Let $f\in \Mor \D, f'\in\Mor\C, m_{\c(f)}\in S(\c(f))$ and $s=\sum_{c\in \obd}g_x,$ with $g_x\in\Id_x$ for any $x\in\obd.$ The right action of $k\S$ on $k\otimes_{k\C}S[\D]$ is given by 
 \[(1\otimes m_{\c(f)}f)\cdot s=1\otimes S(g_{\c(f)}^{\op})(m_{\c(f)})f\circ g_{\d(f)},\]
see Definition \ref{def:44}. Finally, the right action of $k\S$ on $M_S\otimes k\C$ is given by
\[(m_x\otimes f')\cdot s=S(g_x^{\op})(m_x)\otimes f'\circ \s(g_{\d(f)}),\]
where $m_x\in S(x),$ for any $x\in \obd$ and,  with this right $k\S$-module structure, note that

\[\mathrm{IndT}_{\s,\D}^{\C}(S)[\C]=(M_S\otimes k\C)^{\S}.\]
\end{proof}
\begin{remark} Let $\f:\D\rightarrow \C$ be an injective on objects covariant functor. Then  $\f(\D)$ is a subcategory of $\C$ and we have a decomposition:
\[
\xymatrix@C+=2cm{
\D\ar[r]^{\s} & \f(\D) \ar@{^{(}->}[r]^{\i}& \C,}
\]
where $\i$ is the inclusion functor (hence injective on morphisms) and $\s(x)=\f(x),$ where $f(x)\in\Ob\f(\D),$ for all $x\in\obd.$ On morphisms we have $\s(f)=\f(f)$, for all $f\in\Hom_{\D}(x,y),$ hence $\s$ is bijective on objects and surjective on morphisms. 
Our final goal was initially to define Turull's and Puig's induction for any such functor $\f$ and to obtain a similar theorem as Theorem \ref{thm:13} for $\f$. But, for the moment there are too many technical issues for defining  an injective on morphisms Turull's induction of the form $
\Ind^{\C}_{\i,\f(\D)}$ and then to generalize as $\Ind^{\C}_{\f,\D}=\Ind^{\C}_{\i,\f(\D)}\circ \Ind^{\f(\D)}_{\s,\D}.$
\end{remark}

\phantomsection

\end{document}